\documentclass[12pt]{amsart}

\usepackage{amsmath,mathtools,amsthm,amsfonts,amssymb,verbatim,cases,graphicx,multirow}
\usepackage[usenames,dvipsnames]{xcolor}

\usepackage[letterpaper]{geometry}
\usepackage{url}

\newtheorem{theorem}[equation]{Theorem}
\newtheorem{corollary}[equation]{Corollary}

\newtheorem{lemma}[equation]{Lemma}
\newtheorem{proposition}[equation]{Proposition}

\theoremstyle{definition}

\newtheorem{exam}[equation]{Example}
\numberwithin{equation}{section}

\definecolor{mjo}{rgb}{0,0,.9}
\newcommand{\ZZ}{\mathbb{Z}}

\newcommand{\SYM}{\operatorname{Sym}}
\newcommand{\ALT}{\operatorname{Alt}}

\global\long\def\dng{\text{\sf DNG}}
\global\long\def\mex{\operatorname{mex}}
\global\long\def\nim{\operatorname{nim}}
\global\long\def\opt{\operatorname{Opt}}

\global\long\def\pty{\operatorname{pty}}
\global\long\def\type{\operatorname{type}}

\global\long\def\spr{d}

\global\long\def\dih{\operatorname{Dih}}

\newcommand{\notes}[1]{}


\begin{document}

\setlength{\jot}{0pt} 

\title
{Three-person impartial avoidance games for generating finite cyclic, dihedral, and nilpotent groups}

\author{Bret J.~Benesh}
\author{Marisa R. Gaetz}

\address{
Department of Mathematics,
College of Saint Benedict and Saint John's University,
37 College Avenue South,
Saint Joseph, MN 56374-5011, USA
}
\email{bbenesh@csbsju.edu}
\address{
Saint John's Preparatory School,
2280 Water Tower Rd, 
Collegeville, MN 56321-4000, USA
}
\email{marisagaetz@gmail.com}


\date{\today}


\begin{abstract}
We study a three-player variation of the impartial avoidance game introduced by Anderson and Harary.  Three players take turns selecting previously-unselected elements of a finite group.  The losing player is the one who selects an element that causes the set of jointly-selected elements to be a generating set for the group, with the previous player winning and the remaining player coming in second place.  We describe the winning strategy for these games on cyclic, dihedral, and nilpotent groups.
\end{abstract}


\maketitle


\begin{section}{Introduction}


The game \emph{Do Not Generate} was introduced by Anderson and Harary~\cite{anderson.harary:achievement}.  In this game, two players take 
turns selecting previously unselected elements of a finite group until the group is generated by the jointly-selected elements.   The losing player is the first player who selects an element that causes the jointly-selected elements to generate the entire group.  The strategies and nim-numbers for these games were classified in~\cite{BeneshErnstSiebenDNG}.  We will assume that all groups are nontrivial, since \emph{Do Not Generate} cannot be played on the trivial group, as the empty set generates the entire group in this case.

We modify this game to include three players.  The player who generates the group still loses, but we will define the person who plays immediately before the loser to be the winner and the remaining player to be the runner-up.  This ranking gives each player a preference for which of the opponents should win, and a player will help her preferred opponent win if she cannot ensure victory for herself.   This has the effect that a player may have a winning strategy because of what the previous player does, rather than what the player herself does.  

The intuition for the game follows.  Once the game is over, the players will realize that they simply took turns selecting elements from a single maximal subgroup $M$.  The first player will win if $|M| \equiv 1 \mod{3}$ with the third player as runner-up, the second will win if $|M| \equiv 2 \mod{3}$ with the first player as runner-up, and the third will win if $|M|$ is divisible by $3$ with the second player as runner-up.    

The explanations for this game rely on Lagrange's Theorem and Cauchy's Theorem from elementary group theory.  Let $G$ be a finite group and $p$ be a prime number dividing $|G|$.  Cauchy's Theorem guarantees there is an element $x$ of order $p$ in $G$, and Lagrange's Theorem then guarantees that the final maximal subgroup of the game will be divisible by $p$ if $x$ is ever selected.  In particular, Player~3 will be guaranteed victory by Lagrange's Theorem once an element of order $3$ is played, and the existence of such an element is guaranteed by Cauchy's Theorem if $3$ divides the order of the group. 

\end{section}


\begin{section}{Cyclic groups}


We start by considering a cyclic group $G$ with generator $g$ and order $n$. It is well-known that every maximal subgroup $M$ of $G$ has order $n/p$ for some prime $p$ in this case.  So in the case of cyclic groups, Player~1 can pick the final maximal subgroup by selecting $g^p$ to generate a maximal subgroup of order $n/p$ for some prime $p$.   This allows us to conclude the following.

\begin{theorem}\label{thm:DNGCyclic}
Let $G$ be a nontrivial finite cyclic group of order $n$.   

\begin{enumerate} 
\item If $n \equiv 1 \mod {3}$ and there is a prime number $p$ dividing $n$ such that $p \equiv 1 \mod{3}$, then Player~1 has a winning strategy. 
\item If $n \equiv 1 \mod {3}$ and every prime number $p$ dividing $n$ is such that $p \equiv 2 \mod{3}$, then Player~2 has a winning strategy. 
\item If $n \equiv 2 \mod {3}$, then Player~1 has a winning strategy. 
\item If $n \equiv 0 \mod {9}$, then Player~3 has a winning strategy. 
\item If $n \equiv 3 \mod {9}$, then Player~1 has a winning strategy. 
\item If $n \equiv 6 \mod {9}$, then Player~2 has a winning strategy. 
\end{enumerate} 
\end{theorem} 
\begin{proof}
Every maximal subgroup has order $n/p$ for some prime divisor $p$ of $n$.  If $n \equiv 1 \mod{3}$ and there is a prime divisor $p$ of $n$ such that $p \equiv 1 \mod{3}$, then  Player~1 can generate a maximal subgroup $\langle g^p \rangle$ of order $n/p \equiv 1 \mod{3}$ to win.  Similarly, Player~1 can win by choosing $g^p$ for any prime $p \equiv 2 \mod{3}$ if $n \equiv 2 \mod{3}$ (such a prime $p$ must exist, since $n$ would be equivalent to $0$ or $1$ modulo $3$ otherwise), and Player~1 can win by choosing $g^3$ if $n \equiv 3 \mod{9}$.  

If $n \equiv 1 \mod{3}$ and every prime $p$ dividing $n$ is such that $p \equiv 2 \mod{3}$, then every maximal subgroup will have order $n/p \equiv 2 \mod{3}$ for some prime divisor $p$ of $n$.  Therefore, Player~2 will win, regardless of strategy.  If $n \equiv 6 \mod{9}$, then $n/3 \equiv 2 \mod{3}$ and $n/p \equiv 0 \mod{3}$ for every prime divisor $p$ of $n$ that is not $3$.  Then Player~1 will choose $g^3$ to generate a maximal subgroup of order $n/3 \equiv 2 \mod{3}$ so that Player~2 wins and Player~1 is runner-up.  Finally, if $n \equiv 0 \mod{9}$, then $n/p \equiv 0 \mod{3}$ for any prime divisor $p$ of $n$, including $p=3$.  In this case, Player~3 wins.

\end{proof}

\end{section}


\begin{section}{Dihedral groups}

Let $D_n$ denote a dihedral group of order $2n$.  We start with a statement about the maximal subgroups of $D_n$, which follows easily from the well-known fact that every subgroup of a dihedral group is either cyclic or dihedral.

\begin{proposition}\label{prop:DihedralMaximals}
If $M$ is a maximal subgroup of a dihedral group $D_n$, then either $|M|=n$ or $|M|=2n/p$ for some prime $p$ dividing $n$.
\end{proposition}

We can now determine the outcomes of the avoidance game on dihedral groups by considering the orders of maximal subgroups.

\begin{theorem}\label{thm:DNGDihedral}
Let $D_n$ be a dihedral group.  Then 
\begin{enumerate}
\item Player~1 has a winning strategy if and only if $n \equiv 1 \mod{3}$. 
\item Player~2 has a winning strategy if and only if $n \equiv 2 \mod{3} \text{ or } n \equiv 3 \mod{9}$. 
\item Player~3 has a winning strategy if and only if $n \equiv 0 \mod{9} \text{ or } n \equiv 6 \mod{9}$. 
\end{enumerate}
\end{theorem}
\begin{proof}
Let $r$ generate the cyclic subgroup of rotations of $D_n$.  Suppose that $n \equiv 1 \mod{3}$.  Then Player~1 can win by selecting $r$ to generate a subgroup of order $n \equiv 1 \mod{3}$.  Now suppose that $n \equiv 2 \mod{3}$, in which case $n$ must have a prime divisor $q$ such that $q \equiv 2 \mod{3}$.  If Player~1 chooses a rotation, then Player~2 can choose either $r$ (or the identity, if Player~1 chooses $r$) to generate the cyclic maximal subgroup of order $n$.  If Player~1 chooses a reflection $f$, then Player~2 can choose the rotation $r^q$.  Then $\langle f, r^q \rangle \cong D_{n/q}$ is maximal with $|\langle f,r^q \rangle| = 2n/q \equiv 2(2)/2 = 2 \mod{3}$.  Therefore, Player~2 has a winning strategy if $n \equiv 2 \mod{3}$.

Recall that every maximal subgroup of $D_n$ has order $n$, $2n/p$ for some prime divisor $p \ne 3$ of $n$, or $2n/3$ by Proposition~\ref{prop:DihedralMaximals}.   If $n \equiv 0 \mod{9}$, every possible order of a maximal subgroup is divisible by $3$, so Player~3 will win this game regardless of strategy.  If $n \equiv 3 \mod{9}$, then $n \equiv 0 \mod{3}$, $2n/3 \equiv 2 \mod{3}$, and $2n/p \equiv 0 \mod{3}$ for all primes $p$ not equal to $3$ that divide $n$.  Therefore, either Player~2 or Player~3 will win.  Because Player~1 prefers that Player~2 wins, Player~1 can select an element that generates a subgroup of the cyclic rotations of order $n/3$, and Player~2 can select any reflection to generate a maximal subgroup of order $2n/3 \equiv 2 \mod{3}$.  

Finally, assume that $n \equiv 6 \mod{9}$.  Then $n \equiv 0 \mod{3}$, $2n/3 \equiv 1 \mod{3}$, and $2n/p \equiv 0 \mod{3}$ for all primes $p$ dividing $n$ that are not $3$.  Therefore, either Player~1 or Player~3 will win, although Player~1 cannot win on the first turn because maximal subgroups of order $2n/3$ are not cyclic.  Because Player~2 prefers that Player~3 wins, Player~2 will ensure a victory for Player~3.  If Player~1 selects a rotation, Player~2 will select a generating rotation  to create a maximal subgroup of order $n \equiv 0 \mod{3}$.  If Player~1 selects a reflection $f$, then Player~2 will select a rotation $x$ that generates a rotation subgroup of order $2n/p$ for any prime $p$ not equal to $3$.  Then $|\langle f, x \rangle|=2n/p \equiv 2(0)/p \equiv 0 \mod{3}$, ensuring victory for Player~3.
\end{proof}
\end{section}


\begin{section}{Nilpotent groups}

Recall that nilpotent groups are a generalization of abelian groups and have the following properties.

\begin{theorem}\cite[Theorem~8.19]{Isaacs1994}\label{thm:NilpProps}
Let $G$ be a finite group. Then the following are equivalent.
\begin{enumerate}
\item $G$ is nilpotent.
\item Every maximal subgroup of $G$ is normal in $G$.
\item $G$ is isomorphic to a direct product of its Sylow subgroups. 
\end{enumerate}
\end{theorem}

\begin{proposition}\cite[Problem~8.11]{Isaacs1994}\label{prop:NilpMaximalPrime}
If $M$ is a maximal subgroup of a finite nilpotent group $G$, then $|M|=|G|/p$ for some prime divisor $p$ of $|G|$.
\end{proposition}

We will let $d(G)$ denote the size of a minimal generating set of a group $G$.  We can now determine the outcomes of the avoidance game on finite nilpotent groups by considering the orders of maximal subgroups.

\begin{theorem}\label{thm:DNGNilpotent}
Let $G$ be a nontrivial finite nilpotent group, and let $H$ be the direct product of Sylow $q$-groups of $G$ such that $q \equiv 1 \mod{3}$ and $K$ be the direct product of Sylow $r$-groups of $G$ such that $r \equiv 2 \mod{3}$.
\begin{enumerate}
\item Player~1 has a winning strategy in the following cases.
\begin{enumerate}
\item $|G| \equiv 1 \mod{3}$ and $2d(H) \geq d(K)+1$
\item $|G| \equiv 2 \mod{3}$ and $2d(K) \geq d(H)+1$
\item $|G| \equiv 3 \mod{9}$ and $d(G)=1$
\end{enumerate}
\item Player~2 has a winning strategy in the following cases.
\begin{enumerate}
\item $|G| \equiv 1 \mod{3}$ and $2d(H) < d(K)+1$
\item $|G| \equiv 2 \mod{3}$ and $2d(K) < d(H)+1$
\item $|G| \equiv 6 \mod{9}$ and $d(G)\leq 2$
\end{enumerate}
\item Player~3 has a winning strategy in the following cases.
\begin{enumerate}
\item $|G| \equiv 0 \mod{9}$
\item $|G| \equiv 3 \mod{9}$ and $d(G) \geq 2$
\item $|G| \equiv 6 \mod{9}$ and $d(G) \geq 3$
\end{enumerate}
\end{enumerate}
\end{theorem}
\begin{proof}

First, note that these results agree with those from Theorem~\ref{thm:DNGCyclic} in the case that $d(G)=1$.  The only nontrivial observations to make are first that $d(K)=1$ and $d(H) \in \{0,1\}$ if $|G| \equiv 2 \mod{3}$, and second that, when $|G| \equiv 1 \mod{3}$, $d(H)=0$ if and only if every prime number $p$ dividing $G$ is equivalent to $2$ modulo $3$.   So assume that $d(G) \geq 2$.

If $|G| \equiv 0 \mod{9}$, then $|G|/p \equiv 0 \mod{3}$ for all such $p$ dividing $|G|$, including $p=3$.  Therefore, Player~3 has a winning strategy.

Let $T$ denote the Sylow $3$-subgroup of $G$ for the remainder of the proof, and suppose for the remainder of the paragraph that $|G| \equiv 3 \mod{9}$ so that $|T|=3$ and $|H \times K| \equiv 1 \mod{3}$.  This implies that $d(H \times K) \geq 2$, lest $G$ be cyclic.  The maximal subgroups must each have order either $|G|/3 = |H \times K| \equiv 1 \mod{3}$ or $|G|/p \equiv 0\mod{3}$ for some prime $p$ dividing $|G|$.  Thus Player~2 cannot win and will instead help Player~3 win.  Suppose Player~1 selects $x \in G$. If $3$ divides $|\langle x \rangle|$, then Player~3 will win.  Otherwise, Player~2 can select a nontrivial element $t$ of $T$, which is a legal play because $\langle x,t \rangle = \langle t \rangle \times \langle x \rangle < T \times (H \times K) = G$ since $d(H \times K) \geq 2$.  This guarantees victory for Player~3.  

Now suppose that $d(G)=2$ and $|G| \equiv 6 \mod{9}$.  Again $|T|=3$, but now $|H \times K| \equiv 2 \mod{3}$ with $d(H \times K) = 2$, lest $G$ be cyclic.  The maximal subgroups must each have order $|G|/3 \equiv 2 \mod{3}$ or $|G|/p \equiv 0\mod{3}$ for some prime $p$ dividing $|G|$.  Thus Player~1 cannot win and will instead help Player~2 win.  Then Player~1 and Player~2 can each choose a generator of $H \times K$, which will result in a maximal subgroup of order $|H \times K|=|G|/3 \equiv 2 \mod{3}$ since $d(H \times K) =2$.  Therefore, Player~2 has a winning strategy. 

So now suppose that $d(G) \geq 3$ and $|G| \equiv 6 \mod{9}$.  Again, $|T|=3$ and $d(H \times K) \geq 3$ since $d(G) \geq 3$.  Suppose Players~1 and~2 choose the elements $x,y \in G$.  If $3$ divides $|\langle x,y\rangle|$, then Player~3 will win.  Otherwise,  Player~3 can choose a nontrivial element $t$ of $T$ so that $\langle t, x ,y \rangle = \langle t \rangle \times \langle x,y\rangle = T \times \langle x,y \rangle < T \times (H \times K)=G$ since $d(H \times K) \geq 3$. Therefore, Player~3 has not generated $G$, and every maximal subgroup containing $\langle x, y, t\rangle$ will have order divisible by $3$, which guarantees victory for Player~3.

Therefore, we may assume that $3$ does not divide $|G|$ and $T$ is trivial, which means that Player~3 cannot win.  Then Player~3 will help Player~1 and, after the first two elements are selected, Player~1 effectively selects two elements for every element Player~2 selects an element since Player~3 helps Player~1.  The game will end with a maximal subgroup $M$ of order $|G|/p$ for some prime $p$ dividing $|G|$.  If $|G| \equiv 1 \mod{3}$, then Player~1 (and hence, Player~3) wants $p \equiv 1 \mod{3}$ and Player~2 wants $p \equiv 2 \mod{3}$.  Player~1's strategy will be to choose elements to generate $K$, since $K \leq M$ implies that $|G|/p = |M|  \equiv 1 \mod{3}$ for some $p$ corresponding to a Sylow subgroup in $H$.  Player~2's strategy will be to generate $H$ for similar reasons.  Because Player~1 essentially gets to choose two generators of $K$ for every generator of $H$ that Player~2 chooses, we see after some simple algebra that Player~1 (with Player~3's help) will be able to generate $K$ before Player~2 can generate $H$ if and only if $2d(H) \geq d(K)+1$.  A similar argument shows that Player~1 wins when $|G| \equiv 2 \mod{3}$ if and only if $2d(K) \geq d(H)+1$.

\end{proof}

\end{section}


\begin{section}{General Results}


Our first result demonstrates conditions under which Player~3 is guaranteed victory.

\begin{proposition}\label{prop:ManyGenerators}
If $G$ is a finite group such that $3$ divides $|G|$ and $d(G) \geq 4$, then Player~3 has a winning strategy.
\end{proposition}
\begin{proof}
Suppose that Players 1 and 2 select elements $x$ and $y$ on the first two turns.  By Cauchy's Theorem, there is an element $t$ of order $3$ in $G$, and $\langle x,y,t \rangle < G$ since $d(G) \geq 4$.  Then Player~3 wins by selecting $t$, since $t$ will be in the final maximal subgroup $M$ of the game ensuring that $|M| \equiv 0 \mod{3}$.
\end{proof}

Our next result depends on some definitions.  Let $G$ be a noncyclic finite group, and denote the set of maximal subgroups of $G$ by $\mathcal{M}$.  We say that a subset $\mathcal{X}$ of $\mathcal{M}$ \emph{$n$-covers $G$} if for every $n$ elements $g_1,\ldots,g_n \in G$, there is a maximal subgroup $M \in \mathcal{X}$ such that $g_1,\ldots,g_n \in M$.   We let $\mathcal{M}_3$ be the maximal subgroups of $G$ with orders divisible by $3$.   We offer a first approximation for the $3$-player analogue to the main theorems in ~\cite{BeneshErnstSiebenDNG} below.

\begin{theorem}\label{thm:Coverings} 
Let $G$ be a finite group.  Then
\begin{enumerate} 
\item If $\mathcal{M}_3$ $2$-covers $G$, then Player~3 has a winning strategy.
\item\label{item:1Cover} If $\mathcal{M}_3$ $1$-covers $G$ but does not $2$-cover $G$, then Player~1 does not have a winning strategy.
\item\label{item:0Cover} Otherwise, Player~3 does not have a winning strategy.
\end{enumerate} 
\end{theorem} 
\begin{proof}
Suppose that $\mathcal{M}_3$ $2$-covers $G$ and Players 1 and 2 select element $x$ and $y$ of $G$ initially.  Since $\mathcal{M}_3$ $2$-covers $G$, there is a maximal subgroup $M \in \mathcal{M}_3$ such that $x,y \in M$.  Then Player~3 can select any element of order $3$ in $M$ to ensure that the game's final maximal subgroup is divisible by $3$.

Now suppose that $\mathcal{M}_3$ covers $G$ but does not $2$-cover $G$ and Player~1 selects an element $x$ of $G$ initially.  Because $\mathcal{M}_3$ $1$-covers $G$, there is a maximal subgroup $M$ containing $x$.  Player~2 does not want Player~1 to win, so Player~2 will either choose an element that ensures victory for herself or select an element of order $3$ from $M$ to ensure that Player~1 does not win.

Now assume that $\mathcal{M}_3$ does not cover $G$.   Then Player~1 can choose an element not contained in $\cup \mathcal{M}_3$ to ensure that Player~3 does not win.
\end{proof}

\end{section}

\begin{section}{Further Questions}

We close with some open questions.

\begin{enumerate}
\item What are the strategies and outcomes for groups other than cyclic, dihedral, and nilpotent groups?
\item Can we generalize these methods for the analogous game involving $p$ players for some prime $p$?  This will allow us to still use Cauchy's Theorem.
\item Can we generalize these results for the analogous game involving $n$ players for any $n \geq 4$? 
\item What are the winning strategies in the analogous achievement game where the player who generates the group wins, rather than loses?
\item Can we improve the statements of Items~\ref{item:1Cover} and~\ref{item:0Cover} of Theorem~\ref{thm:Coverings} to determine exactly which player wins in each case? 
\end{enumerate}

\end{section}

\bibliographystyle{amsplain}
\bibliography{game}

\providecommand{\bysame}{\leavevmode\hbox to3em{\hrulefill}\thinspace}
\providecommand{\MR}{\relax\ifhmode\unskip\space\fi MR }
\providecommand{\MRhref}[2]{%
  \href{http://www.ams.org/mathscinet-getitem?mr=#1}{#2}
}
\providecommand{\href}[2]{#2}
\begin{thebibliography}{1}

\bibitem{anderson.harary:achievement}
M.~Anderson and F.~Harary, \emph{Achievement and avoidance games for generating
  abelian groups}, Internat. J. Game Theory \textbf{16} (1987), no.~4,
  321--325.

\bibitem{BeneshErnstSiebenDNG}
B.J. Benesh, D.C. Ernst, and N.~Sieben, \emph{Impartial avoidance games for
  generating finite groups}, North-Western European Journal of Mathematics
  \textbf{2} (2016), 83--102.

\bibitem{Isaacs1994}
I.M. Isaacs, \emph{Algebra: a graduate course}, vol. 100, American Mathematical
  Society, Providence, RI, 2009, Reprint of the 1994 original.

\end{thebibliography}
\end{document}